\documentclass{article}%
\usepackage{amsfonts}
\usepackage{graphicx}
\usepackage{amsmath}%
\usepackage{amssymb}
\providecommand{\U}[1]{\protect\rule{.1in}{.1in}}
\newtheorem{theorem}{Theorem}

\newtheorem{corollary}[theorem]{Corollary}

\newtheorem{lemma}[theorem]{Lemma}
\newtheorem{notation}[theorem]{Notation}

\newtheorem{proposition}[theorem]{Proposition}
\newtheorem{remark}[theorem]{Remark}

\newenvironment{proof}[1][Proof]{\textbf{#1.} }{\ \rule{0.5em}{0.5em}}
\begin{document}

\title{Axiomatic Differential Geometry II-1\\-Its Developments-\\Chapter 1:Vector Fields}
\author{Hirokazu Nishimura\\Institute of Mathematics\\University of Tsukuba\\Tsukuba, Ibaraki, 305-8571, JAPAN}
\maketitle

\begin{abstract}
In our previous paper entitled ''Axiomatic differential geometry -towards
model categories of differential geometry-, we have given a category-theoretic
framework of differential geometry. As the first part of our series of papers
concerned with differential-geometric developments within the above axiomatic
scheme, this paper is devoted to vector fields. The principal result is that
the totality of vector fields on a microlinear and Weil exponential object
forms a Lie algebra.

\end{abstract}

\section{Introduction\label{s1}}

In \cite{nishi2} we have given a skeleton of our axiomatic approach to
differential geometry. This paper, concerned with vector fields, is the first
part of our theoretical developments within the axiomatic framework.
Subsequent papers deal with differential forms, the Fr\"{o}licher-Nijenhuis
calculus, jet bundles, connections and so on.

In Section \ref{s2}, we develop a convenient system of locutions in speaking
of Weil algebras. Since we are no longer allowed to speak elementwise in our
general context, we have to learn how to express ''the tangent space is a
linear space'', to say nothing of how to prove it, which will be done in
Section \ref{s3}. The principal result of Section \ref{s4} is that the
totality of vector fields forms a Lie algebra.

\section{Preliminaries\label{s2}}

\subsection{Weil Algebras and Infinitesimal Objects}

Let $k$ be a commutative ring. As in our previous paper, we denote by
$\mathbf{Weil}_{k}$ the category of Weil $k$-algebras. Roughly speaking, each
Weil $k$-algebra corresponds to an infinitesimal object in the shade. By way
of example, the Weil algebra $k[X]/(X^{2})$ (=the quotient ring of the
polynomial ring $k[X]$\ of an indeterminate $X$\ over $k$ modulo the ideal
$(X^{2})$\ generated by $X^{2}$) corresponds to the infinitesimal object of
first-order nilpotent infinitesimals, while the Weil algebra $k[X]/(X^{3})$
corresponds to the infinitesimal object of second-order nilpotent
infinitesimals. Although an infinitesimal object is undoubtedly imaginary in
our real world, as has harassed both mathematicians and philosophers of the
17th and the 18th centuries such as extravagantly skeptical philosopher
Berkley (because mathematicians at that time preferred to talk infinitesimal
objects as if they were real entities), each Weil algebra yields its
corresponding \textit{Weil functor} or \textit{Weil prolongation} in our real
world. By way of example, the Weil algebra $k[X]/(X^{2})$ yields the tangent
bundle functor as its corresponding Weil functor. Weil functors are major
players in our axiomatic approach.

\textit{Synthetic differential geometry }(usually abbreviated to SDG), which
is a kind of differential geometry with a cornucopia of nilpotent
infinitesimals, was forced to invent its models, in which nilpotent
infinitesimals were transparently visible. For a standard textbook on SDG, the
reader is referred to \cite{lav}, while he or she is referred to \cite{ko} for
the model theory of SDG. Although we do not get involved in SDG herein, we
will exploit locutions in terms of infinitesimal objects so as to make the
paper highly readable. Thus we prefer to write $\mathcal{W}_{D}$\ and
$\mathcal{W}_{D_{2}}$\ in place of $k[X]/(X^{2})$ and $k[X]/(X^{3})$
respectively, where $D$ stands for the infinitesimal object of first-order
nilpotent infinitesimals, and $D_{2}$\ stands for the infinitesimal object of
second-order nilpotent infinitesimals. More generally, given a natural number
$n$, we denote by $D_{n}$ the infinitesimal object corresponding to the Weil
$k$-algebra $k[X]/(X^{n+1})$. Obviously we have $D_{1}=D$. Even more
generally, given natural numbers $m,n$, we denote by $D(m)_{n}$ the
infinitesimal object corresponding to the Weil algebra $k[X_{1},...,X_{m}]/I$,
where $I$ is the ideal generated by $X_{i_{1}}...X_{i_{n+1}}$'s with
$i_{1},...,i_{n+1}$ being integers such that $1\leq i_{1},...,i_{n+1}\leq m$.
Therefore we have $D(1)_{n}=D_{n}$, while we write $D\left(  m\right)  $ for
$D\left(  m\right)  _{1}$.We will write $\mathcal{W}_{d\in D_{2}\mapsto
d^{2}\in D}$, by way of example, for the homomorphim of Weil algebras
$k[X]/(X^{2})\rightarrow k[X]/(X^{3})$ induced by the homomorphism
$X\rightarrow X^{2}$ of the polynomial ring \ $k[X]$ to itself. Such locutions
are justifiable, because the category $\mathbf{Weil}_{k}$ of Weil $k$-algebras
in the real world and the category of infinitesimal objects in the shade are
dual to each other in a sense. Thus we have a contravariant functor
$\mathcal{W}$\ from the category of infinitesimal objects in the shade to the
category of Weil algebras in the real world, yielding a contravariant
equivalence between the two categories.

\subsection{Assumptions}

We fix a DG-category $\left(  \mathcal{K},\mathbb{R},\mathbf{T},\alpha\right)
$\ in the sense of \cite{nishi2}, while $M$ is a microlinear and Weil
exponentiable object in $\mathcal{K}$.

\section{\label{s3}Tangent Bundles}

\begin{proposition}
\label{t3.1}If $M$ is a microlinear object in $\mathcal{K}$, then we have
\[
M\otimes\mathcal{W}_{D(2)}=(M\otimes\mathcal{W}_{D})\times_{M}(M\otimes
\mathcal{W}_{D})
\]

\end{proposition}

\begin{proof}
We have the following pullback diagram of Weil $k$-algebras:
\begin{equation}%
\begin{array}
[c]{ccc}%
\mathcal{W}_{D(2)} & \rightarrow & \mathcal{W}_{D}\\
\downarrow &  & \downarrow\\
\mathcal{W}_{D} & \rightarrow & \mathcal{W}_{1}%
\end{array}
\label{3.1.1}%
\end{equation}
where the left vertical arrow is
\[
\mathcal{W}_{d\in D\mapsto(d,0)\in D(2)}%
\]
while the upper horizontal arrow is
\[
\mathcal{W}_{d\in D\mapsto(0,d)\in D(2)}%
\]
The above pullback diagram naturally gives rise to the following pullback
diagram because of the microlinearity of $M$:
\[%
\begin{array}
[c]{ccc}%
M\otimes\mathcal{W}_{D(2)} & \rightarrow & M\otimes\mathcal{W}_{D}\\
\downarrow &  & \downarrow\\
M\otimes\mathcal{W}_{D} & \rightarrow & M\otimes\mathcal{W}_{1}=M
\end{array}
\]
This completes the proof.
\end{proof}

\begin{corollary}
The canonical projection $\tau_{\mathcal{W}_{D(2)}}\left(  M\right)
:M\otimes\mathcal{W}_{D(2)}\rightarrow M$ is a product of two copies of
$\tau_{\mathcal{W}_{D}}\left(  M\right)  :M\otimes\mathcal{W}_{D}\rightarrow M
$ in the slice category $\mathcal{K}/M$.
\end{corollary}

Now we are in a position to define basic operations on $\tau_{\mathcal{W}_{D}%
}\left(  M\right)  :M\otimes\mathcal{W}_{D}\rightarrow M$ in the slice
category $\mathcal{K}/M$ so as to make it a $k$-module.

\begin{enumerate}
\item The addition is defined by
\[
\mathrm{id}_{M}\otimes\mathcal{W}_{+_{D}}:M\otimes\mathcal{W}_{D(2)}%
\rightarrow M\otimes\mathcal{W}_{D}%
\]
where the fabulous mapping $+_{D}:D\rightarrow D(2)$ is
\[
+_{D}:d\in D\mapsto(d,d)\in D(2)
\]

\item The identity with respect to the above addition is defined by
\[
\mathrm{id}_{M}\otimes\mathcal{W}_{0_{D}}:M=M\otimes\mathcal{W}_{1}\rightarrow
M\otimes\mathcal{W}_{D}%
\]
where the fabulous mapping $0_{D}:D\rightarrow1$ is the unique mapping.

\item The inverse with respect to the above addition is defined by
\[
\mathrm{id}_{M}\otimes\mathcal{W}_{-_{D}}:M\otimes\mathcal{W}_{D}\rightarrow
M\otimes\mathcal{W}_{D}%
\]
where the fabulous mapping $-_{D}:D\rightarrow D$ is
\[
-_{D}:d\in D\mapsto-d\in D
\]

\item The scalar multiplication by a scalar $\xi\in k$ is defined by
\[
\mathrm{id}_{M}\otimes\mathcal{W}_{\xi_{D}}:M\otimes\mathcal{W}_{D}\rightarrow
M\otimes\mathcal{W}_{D}%
\]
where the fabulous mapping $\xi_{D}:D\rightarrow D$ is
\[
\xi_{D}:d\in D\mapsto\xi d\in D
\]

\end{enumerate}

Now we have

\begin{theorem}
\label{t3.2}The canonical projection $\tau_{\mathcal{W}_{D}}\left(  M\right)
:M\otimes\mathcal{W}_{D}\rightarrow M$ is a $k$-module in the slice category
$\mathcal{K}/M$.
\end{theorem}

\begin{proof}
\begin{enumerate}
\item The associativity of the addition follows from the following commutative
diagram:
\[%
\begin{array}
[c]{ccccc}
&  & \mathrm{id}_{M}\otimes\mathcal{W}_{\epsilon_{23}} &  & \\
& M\otimes\mathcal{W}_{D(3)} & \rightarrow & M\otimes\mathcal{W}_{D(2)} & \\
\mathrm{id}_{M}\otimes\mathcal{W}_{\epsilon_{12}} & \downarrow &  & \downarrow
& \mathrm{id}_{M}\otimes\mathcal{W}_{+_{D}}\\
& M\otimes\mathcal{W}_{D(2)} & \rightarrow & M\otimes\mathcal{W}_{D} & \\
&  & \mathrm{id}_{M}\otimes\mathcal{W}_{+_{D}} &  &
\end{array}
\]
where the fabulous mapping $\epsilon_{23}:D(2)\rightarrow D(3)$ is
\[
(d_{1},d_{2})\in D(2)\mapsto(d_{1},d_{1},d_{2})\in D(3)
\]
while the fabulous mapping $\epsilon_{12}:D(2)\rightarrow D(3)$ is
\[
(d_{1},d_{2})\in D(2)\mapsto(d_{1},d_{2},d_{2})\in D(3)
\]

\item The commutativity of the addition follows readily from the commutative
diagram
\[%
\begin{array}
[c]{cccc}
& \mathcal{W}_{+_{D}} &  & \\
\mathcal{W}_{D} & \leftarrow & \mathcal{W}_{D(2)} & \\
& \nwarrow & \uparrow & \mathcal{W}_{\tau}\\
\mathcal{W}_{+_{D}} &  & \mathcal{W}_{D(2)} &
\end{array}
\]
where the fabulous mapping $\tau:D(2)\rightarrow D(2)$ is
\[
(d_{1},d_{2})\in D(2)\mapsto(d_{2},d_{1})\in D(2)\text{.}%
\]

\item To see that the identity defined above really plays the identity with
respect to the above addition, it suffices to note that the composition of the
following two fabulous mappings
\[
d\in D\mapsto(d,0)\in D(2)
\]
\[
(d_{1},d_{2})\in D(2)\mapsto d_{1}\in D
\]
in order is the identity mapping of $D$, while the composition of the
following two fabulous mappings
\[
d\in D\mapsto(0,d)\in D(2)
\]
\[
(d_{1},d_{2})\in D(2)\mapsto d_{1}\in D
\]
in order is the constant mapping
\[
d\in D\mapsto0\in D
\]

\item To see that the addition of scalars distributes with respect to the
scalar multiplication, it suffices to note that, for any $\xi_{1},\xi_{2}\in
k$, the composition of the following two fabulous mappings
\[
d\in D\mapsto(d,0)\in D(2)
\]
\[
(d_{1},d_{2})\in D(2)\mapsto\xi_{1}d_{1}+\xi_{2}d_{2}\in D
\]
in order is the mapping
\[
d\in D\mapsto\xi_{1}d_{1}\in D
\]
and the composition of the two fabulous mappings
\[
d\in D\mapsto(0,d)\in D(2)
\]
\[
(d_{1},d_{2})\in D(2)\mapsto\xi_{1}d_{1}+\xi_{2}d_{2}\in D
\]
in order is the mapping
\[
d\in D\mapsto\xi_{2}d_{2}\in D
\]
while the composition of the two fabulous mappings
\[
d\in D\mapsto(d,d)\in D(2)
\]
\[
(d_{1},d_{2})\in D(2)\mapsto\xi_{1}d_{1}+\xi_{2}d_{2}\in D
\]
in order is no other than the mapping
\[
d\in D\mapsto(\xi_{1}+\xi_{2})d\in D
\]

\item To see that the addition of vectors distributes with respect to the
scalar multiplication, it suffices to note that, for any $\xi\in k$, the
composition of the two fabulous mappings
\[
d\in D\mapsto(d,0)\in D(2)
\]
\[
(d_{1},d_{2})\in D(2)\mapsto(\xi d_{1},\xi d_{2})\in D(2)
\]
in order is the composition of the two fabulous mappings
\[
d\in D\mapsto\xi d\in D
\]
\[
d\in D\mapsto(d,0)\in D(2)
\]
in order, and the composition of the two fabulous mappings
\[
d\in D\mapsto(0,d)\in D(2)
\]
\[
(d_{1},d_{2})\in D(2)\mapsto(\xi d_{1},\xi d_{2})\in D(2)
\]
in order is the composition of the two fabulous mappings
\[
d\in D\mapsto\xi d\in D
\]
\[
d\in D\mapsto(0,d)\in D(2)
\]
in order, while the composition of the two fabulous mappings
\[
d\in D\mapsto(d,d)\in D(2)
\]
\[
(d_{1},d_{2})\in D(2)\mapsto(\xi d_{1},\xi d_{2})\in D(2)
\]
in order is no other than the composition of the two fabulous mappings
\[
d\in D\mapsto\xi d\in D
\]
\[
d\in D\mapsto(d,d)\in D(2)
\]
in order.

\item The verification of the other axioms for being a $k$-module can safely
be left to the reader.
\end{enumerate}
\end{proof}

\begin{remark}
Given a morphims $f:N\rightarrow M$ in $\mathcal{K}$, the pullback functor
$f^{\ast}:\mathcal{K}/M\rightarrow\mathcal{K}/N$ is left exact ( preserving
finite products, in particular), so that
\[
f^{\ast}\left(
\begin{array}
[c]{c}%
M\otimes\mathcal{W}_{D}\\%
\begin{array}
[c]{ccc}%
\tau_{\mathcal{W}_{D}}(M) & \downarrow & \,\quad\quad\quad
\end{array}
\\
M
\end{array}
\right)
\]
is a $k$-module in $\mathcal{K}/N$. In case that $N=1$, we have $\mathcal{K}%
/1=\mathcal{K}$ yielding the notion of the tangent space $\left(
M\otimes\mathcal{W}_{D}\right)  _{x}$\ of $M$\ at a point $x:1\rightarrow M$.
\end{remark}

\section{\label{s4}The Lie Algebra of Vector Fields}

The totality of vector fields on $M$\ can be delineated in two distinct ways.
First we must fix our notation.

\begin{notation}
We write $\left(  (M\otimes\mathcal{W}_{D})^{M}\right)  _{\mathrm{id}_{M}}$
for the pullback of
\[%
\begin{array}
[c]{ccc}%
\left(  (M\otimes\mathcal{W}_{D})^{M}\right)  _{\mathrm{id}_{M}} & \rightarrow
& (M\otimes\mathcal{W}_{D})^{M}\\
\downarrow &  & \downarrow\\
1 & \rightarrow & M^{M}%
\end{array}
\]
where the right vertical arrow is
\[
\left(  \tau_{\mathcal{W}_{D}}\left(  M\right)  \right)  ^{M}:(M\otimes
\mathcal{W}_{D})^{M}\rightarrow M^{M}\text{,}%
\]
while the bottom horizontal arrow is the exponential transpose of
\[
\mathrm{id}_{M}:M=1\times M\rightarrow M\text{.}%
\]

\end{notation}

\begin{notation}
We write $(M^{M}\otimes\mathcal{W}_{D})_{\mathrm{id}_{M}}$ for the pullback
of
\[%
\begin{array}
[c]{ccc}%
(M^{M}\otimes\mathcal{W}_{D})_{\mathrm{id}_{M}} & \rightarrow & M^{M}%
\otimes\mathcal{W}_{D}\\
\downarrow &  & \downarrow\\
1 & \rightarrow & M^{M}%
\end{array}
\]
where the right vertical arrow is
\[
\tau_{M^{M}}:M^{M}\otimes\mathcal{W}_{D}\rightarrow M^{M}\text{,}%
\]
while the bottom horizontal arrow is the exponential transpose of
\[
\mathrm{id}_{M}:M=1\times M\rightarrow M\text{.}%
\]

\end{notation}

\begin{theorem}
\label{t4.1}The object $\left(  (M\otimes\mathcal{W}_{D})^{M}\right)
_{\mathrm{id}_{M}}$ can naturally be identified with the object $(M^{M}%
\otimes\mathcal{W}_{D})_{\mathrm{id}_{M}}$.
\end{theorem}

\begin{proof}
It suffices to note that the diagram
\[%
\begin{array}
[c]{ccc}%
(M\otimes\mathcal{W}_{D})^{M} & = & M^{M}\otimes\mathcal{W}_{D}\\
\searrow &  & \swarrow\\
& M^{M} &
\end{array}
\]
is commutaive, where the right slant arrow is $\tau_{\mathcal{W}_{D}}\left(
M^{M}\right)  :M^{M}\otimes\mathcal{W}_{D}\rightarrow M^{M}$, while the left
slant arrow is $\left(  \tau_{\mathcal{W}_{D}}\left(  M\right)  \right)
^{M}:(M\otimes\mathcal{W}_{D})^{M}\rightarrow M^{M}$.
\end{proof}

\begin{remark}
Thus the totality of vector fields on $M$\ is represented by $(M^{M}%
\otimes\mathcal{W}_{D})_{\mathrm{id}_{M}}$ as well as by $\left(
(M\otimes\mathcal{W}_{D})^{M}\right)  _{\mathrm{id}_{M}}$. The first
viewpoint, which reckons vector fields as the tangent space to $M$\ at the
identitity transformation, is preferred in this paper, while the second
viewpoint, which regards vector fields on $M$\ as sections of the tangent
bundle $\tau_{\mathcal{W}_{D}}(M):M\otimes\mathcal{W}_{D}\rightarrow M$, has
been orthodox in conventional differential geomety.
\end{remark}

\begin{notation}
We write $\mathrm{ass}_{M}:M^{M}\times M^{M}\rightarrow M^{M}$ for the
morphism obtained as the exponential transpose of the composition of
\[
\mathrm{ev}_{M}\times\mathrm{id}_{M^{M}}:M\times M^{M}\times M^{M}=\left(
M\times M^{M}\right)  \times M^{M}\rightarrow M\times M^{M}%
\]
and
\[
\mathrm{ev}_{M}:M\times M^{M}\rightarrow M
\]
where $\mathrm{ev}_{M}:M\times M^{M}\rightarrow M$ stands for the evaluation
morphism. Similarly we write $\overline{\mathrm{ass}}_{M}:M^{M}\times
M^{M}\rightarrow M^{M}$ for the morphism obtained as the exponential transpose
of the composition of
\begin{align*}
\mathrm{id}_{M^{M}}\times\mathrm{ev}_{M}  & :M\times M^{M}\times M^{M}%
=M^{M}\times M\times M^{M}=M^{M}\times\left(  M\times M^{M}\right) \\
& \rightarrow M^{M}\times M=M\times M^{M}%
\end{align*}
and
\[
\mathrm{ev}_{M}:M\times M^{M}\rightarrow M
\]

\end{notation}

\begin{notation}
We write $\mathrm{Ass}_{M}^{m,n}:\left(  M^{M}\otimes\mathcal{W}_{D^{m}%
}\right)  \times\left(  M^{M}\otimes\mathcal{W}_{D^{n}}\right)  \rightarrow
M^{M}\otimes\mathcal{W}_{D^{m+n}}$ for the morphism obtained as the
composition of
\begin{align*}
& \left(  \mathrm{id}_{M^{M}}\otimes\mathcal{W}_{\left(  d_{1},...,d_{m}%
,d_{m+1},...,d_{m+n}\right)  \in D^{m+n}\mapsto\left(  d_{1},...,d_{m}\right)
\in D^{m}}\right)  \times\\
& \left(  \mathrm{id}_{M^{M}}\otimes\mathcal{W}_{\left(  d_{1},...,d_{m}%
,d_{m+1},...,d_{m+n}\right)  \in D^{m+n}\mapsto\left(  d_{m+1},...,d_{m+n}%
\right)  \in D^{n}}\right) \\
& :\left(  M^{M}\otimes\mathcal{W}_{D^{m}}\right)  \times\left(  M^{M}%
\otimes\mathcal{W}_{D^{n}}\right)  \rightarrow\left(  M^{M}\otimes
\mathcal{W}_{D^{m+n}}\right)  \times\left(  M^{M}\otimes\mathcal{W}_{D^{m+n}%
}\right) \\
& =\left(  M^{M}\times M^{M}\right)  \otimes\mathcal{W}_{D^{m+n}}%
\end{align*}
\[
\mathrm{ass}_{M}\otimes\mathrm{id}_{\mathcal{W}_{D^{m+n}}}:\left(  M^{M}\times
M^{M}\right)  \otimes\mathcal{W}_{D^{m+n}}\rightarrow M^{M}\otimes
\mathcal{W}_{D^{m+n}}%
\]
in succession.
\end{notation}

\begin{lemma}
\label{t4.2'}We have such laws of associativity as
\[
\mathrm{ass}_{M}\circ\left(  \mathrm{ass}_{M}\times\mathrm{id}_{M^{M}}\right)
=\mathrm{ass}_{M}\circ\left(  \mathrm{id}_{M^{M}}\times\mathrm{ass}%
_{M}\right)
\]
\[
\overline{\mathrm{ass}}_{M}\circ\left(  \overline{\mathrm{ass}}_{M}%
\times\mathrm{id}_{M^{M}}\right)  =\overline{\mathrm{ass}}_{M}\circ\left(
\mathrm{id}_{M^{M}}\times\overline{\mathrm{ass}}_{M}\right)
\]

\end{lemma}

\begin{proposition}
\label{t4.2}We have such laws of associativity as follows:
\begin{align}
& \mathrm{Ass}_{M}^{l+m,n}\circ\left(  \mathrm{Ass}_{M}^{l,m}\times
\mathrm{id}_{M^{M}\otimes\mathcal{W}_{D^{n}}}\right) \nonumber\\
& =\mathrm{Ass}_{M}^{l,m+n}\circ\left(  \mathrm{id}_{M^{M}\otimes
\mathcal{W}_{D^{l}}}\times\mathrm{Ass}_{M}^{m,n}\right) \label{4.2.0.a}%
\end{align}
and
\begin{align}
& \overline{\mathrm{Ass}}_{M}^{l+m,n}\circ\left(  \overline{\mathrm{Ass}}%
_{M}^{l,m}\times\mathrm{id}_{M^{M}\otimes\mathcal{W}_{D^{n}}}\right)
\nonumber\\
& =\overline{\mathrm{Ass}}_{M}^{l,m+n}\circ\left(  \mathrm{id}_{M^{M}%
\otimes\mathcal{W}_{D^{l}}}\times\overline{\mathrm{Ass}}_{M}^{m,n}\right)
\label{4.2.0.b}%
\end{align}

\end{proposition}

\begin{proof}
Here we deal only with the former, leaving the latter to the reader. Now we
have to show that the composition of
\begin{align}
\mathrm{Ass}_{M}^{l,m}\times\mathrm{id}_{M^{M}\otimes\mathcal{W}_{D^{n}}}  &
:\left(  M^{M}\otimes\mathcal{W}_{D^{l}}\right)  \times\left(  M^{M}%
\otimes\mathcal{W}_{D^{m}}\right)  \times\left(  M^{M}\otimes\mathcal{W}%
_{D^{n}}\right) \nonumber\\
& =\left(  \left(  M^{M}\otimes\mathcal{W}_{D^{l}}\right)  \times\left(
M^{M}\otimes\mathcal{W}_{D^{m}}\right)  \right)  \times\left(  M^{M}%
\otimes\mathcal{W}_{D^{n}}\right) \nonumber\\
& \rightarrow\left(  M^{M}\otimes\mathcal{W}_{D^{l+m}}\right)  \times\left(
M^{M}\otimes\mathcal{W}_{D^{n}}\right) \label{4.2.20}%
\end{align}
followed by
\begin{equation}
\mathrm{Ass}_{M}^{l+m,n}:\left(  M^{M}\otimes\mathcal{W}_{D^{l+m}}\right)
\times\left(  M^{M}\otimes\mathcal{W}_{D^{n}}\right)  \rightarrow M^{M}%
\otimes\mathcal{W}_{D^{l+m+n}}\label{4.2.21}%
\end{equation}
is equal to that of
\begin{align}
\mathrm{id}_{M^{M}\otimes\mathcal{W}_{D^{l}}}\times\mathrm{Ass}_{M}^{m,n}  &
:\left(  M^{M}\otimes\mathcal{W}_{D^{l}}\right)  \times\left(  M^{M}%
\otimes\mathcal{W}_{D^{m}}\right)  \times\left(  M^{M}\otimes\mathcal{W}%
_{D^{n}}\right) \nonumber\\
& =\left(  M^{M}\otimes\mathcal{W}_{D^{l}}\right)  \times\left(  \left(
M^{M}\otimes\mathcal{W}_{D^{m}}\right)  \times\left(  M^{M}\otimes
\mathcal{W}_{D^{n}}\right)  \right) \nonumber\\
& \rightarrow\left(  M^{M}\otimes\mathcal{W}_{D^{l}}\right)  \times\left(
M^{M}\otimes\mathcal{W}_{D^{m+n}}\right) \label{4.2.22}%
\end{align}
followed by
\begin{equation}
\mathrm{Ass}_{M}^{l,m+n}:\left(  M^{M}\otimes\mathcal{W}_{D^{l}}\right)
\times\left(  M^{M}\otimes\mathcal{W}_{D^{m+n}}\right)  \rightarrow
M^{M}\otimes\mathcal{W}_{D^{l+m+n}}\label{4.2.23}%
\end{equation}
The former is the composition of
\begin{align}
& \left(  M^{M}\otimes\mathcal{W}_{D^{l}}\right)  \times\left(  M^{M}%
\otimes\mathcal{W}_{D^{m}}\right)  \times\left(  M^{M}\otimes\mathcal{W}%
_{D^{n}}\right) \nonumber\\
& =\left(  \left(  M^{M}\otimes\mathcal{W}_{D^{l}}\right)  \times\left(
M^{M}\otimes\mathcal{W}_{D^{m}}\right)  \right)  \times\left(  M^{M}%
\otimes\mathcal{W}_{D^{n}}\right) \nonumber\\
& \rightarrow\left(  \left(  M^{M}\otimes\mathcal{W}_{D^{l+m}}\right)
\times\left(  M^{M}\otimes\mathcal{W}_{D^{l+m}}\right)  \right)  \times\left(
M^{M}\otimes\mathcal{W}_{D^{n}}\right) \nonumber\\
& =\left(  \left(  M^{M}\times M^{M}\right)  \otimes\mathcal{W}_{D^{l+m}%
}\right)  \times\left(  M^{M}\otimes\mathcal{W}_{D^{n}}\right)  \text{,}%
\label{4.2.1}%
\end{align}
\begin{align}
\left(  \mathrm{ass}_{M}\otimes\mathrm{id}_{\mathcal{W}_{D^{l+m}}}\right)
\times\mathrm{id}_{M^{M}\otimes\mathcal{W}_{D^{n}}}  & :\left(  \left(
M^{M}\times M^{M}\right)  \otimes\mathcal{W}_{D^{l+m}}\right)  \times\left(
M^{M}\otimes\mathcal{W}_{D^{n}}\right) \nonumber\\
& \rightarrow\left(  M^{M}\otimes\mathcal{W}_{D^{l+m}}\right)  \times\left(
M^{M}\otimes\mathcal{W}_{D^{n}}\right)  \text{,}\label{4.2.2}%
\end{align}
\begin{align}
& \left(  M^{M}\otimes\mathcal{W}_{D^{l+m}}\right)  \times\left(  M^{M}%
\otimes\mathcal{W}_{D^{n}}\right) \nonumber\\
& \rightarrow\left(  M^{M}\otimes\mathcal{W}_{D^{l+m+n}}\right)  \times\left(
M^{M}\otimes\mathcal{W}_{D^{l+m+n}}\right) \nonumber\\
& =\left(  M^{M}\times M^{M}\right)  \otimes\mathcal{W}_{D^{l+m+n}%
}\label{4.2.3}%
\end{align}
and
\begin{equation}
\mathrm{ass}_{M}\otimes\mathrm{id}_{\mathcal{W}_{D^{l+m+n}}}:\left(
M^{M}\times M^{M}\right)  \otimes\mathcal{W}_{D^{l+m+n}}\rightarrow
M^{M}\otimes\mathcal{W}_{D^{l+m+n}}\label{4.2.4}%
\end{equation}
in succession, while the latter is the composition of
\begin{align}
& \left(  M^{M}\otimes\mathcal{W}_{D^{l}}\right)  \times\left(  M^{M}%
\otimes\mathcal{W}_{D^{m}}\right)  \times\left(  M^{M}\otimes\mathcal{W}%
_{D^{n}}\right) \nonumber\\
& =\left(  M^{M}\otimes\mathcal{W}_{D^{l}}\right)  \times\left(  \left(
M^{M}\otimes\mathcal{W}_{D^{m}}\right)  \times\left(  M^{M}\otimes
\mathcal{W}_{D^{n}}\right)  \right) \nonumber\\
& \rightarrow\left(  M^{M}\otimes\mathcal{W}_{D^{l}}\right)  \times\left(
\left(  M^{M}\otimes\mathcal{W}_{D^{m+n}}\right)  \times\left(  M^{M}%
\otimes\mathcal{W}_{D^{m+n}}\right)  \right) \nonumber\\
& =\left(  M^{M}\otimes\mathcal{W}_{D^{l}}\right)  \times\left(  \left(
M^{M}\times M^{M}\right)  \otimes\mathcal{W}_{D^{m+n}}\right)  \text{,}%
\label{4.2.5}%
\end{align}
\begin{align}
\mathrm{id}_{M^{M}\otimes\mathcal{W}_{D^{l}}}\times\left(  \mathrm{ass}%
_{M}\otimes\mathrm{id}_{\mathcal{W}_{D^{m+n}}}\right)   & :\left(
M^{M}\otimes\mathcal{W}_{D^{l}}\right)  \times\left(  \left(  M^{M}\times
M^{M}\right)  \otimes\mathcal{W}_{D^{m+n}}\right) \nonumber\\
& \rightarrow\left(  M^{M}\otimes\mathcal{W}_{D^{l}}\right)  \times\left(
M^{M}\otimes\mathcal{W}_{D^{m+n}}\right)  \text{,}\label{4.2.6}%
\end{align}
\begin{align}
& \left(  M^{M}\otimes\mathcal{W}_{D^{l+m}}\right)  \times\left(  M^{M}%
\otimes\mathcal{W}_{D^{n}}\right) \nonumber\\
& \rightarrow\left(  M^{M}\otimes\mathcal{W}_{D^{l+m+n}}\right)  \times\left(
M^{M}\otimes\mathcal{W}_{D^{l+m+n}}\right) \nonumber\\
& =\left(  M^{M}\times M^{M}\right)  \otimes\mathcal{W}_{D^{l+m+n}%
}\label{4.2.7}%
\end{align}
and (\ref{4.2.4}) in succession. Since the composition of (\ref{4.2.2}) and
(\ref{4.2.3}) in succession is equal to that of
\begin{align}
& \left(  \left(  M^{M}\times M^{M}\right)  \otimes\mathcal{W}_{D^{l+m}%
}\right)  \times\left(  M^{M}\otimes\mathcal{W}_{D^{n}}\right) \nonumber\\
& \rightarrow\left(  \left(  M^{M}\times M^{M}\right)  \otimes\mathcal{W}%
_{D^{l+m+n}}\right)  \times\left(  M^{M}\otimes\mathcal{W}_{D^{l+m+n}}\right)
\nonumber\\
& =\left(  \left(  M^{M}\times M^{M}\right)  \times M^{M}\right)
\otimes\mathcal{W}_{D^{l+m+n}}\label{4.2.8}%
\end{align}
and
\begin{align}
\left(  \mathrm{ass}_{M}\times\mathrm{id}_{M^{M}}\right)  \otimes
\mathrm{id}_{\mathcal{W}_{D^{l+m+n}}}  & :\left(  \left(  M^{M}\times
M^{M}\right)  \times M^{M}\right)  \otimes\mathcal{W}_{D^{l+m+n}}\nonumber\\
& \rightarrow\left(  M^{M}\times M^{M}\right)  \otimes\mathcal{W}_{D^{l+m+n}%
}\label{4.2.9}%
\end{align}
in succession, we conclude that the left-hand side of (\ref{4.2.0.a}) is equal
to the composition of (\ref{4.2.1}), (\ref{4.2.8}), (\ref{4.2.9})and
(\ref{4.2.4}) in succession, which is in turn equal to the composition of
\begin{align}
& \left(  M^{M}\otimes\mathcal{W}_{D^{l}}\right)  \times\left(  M^{M}%
\otimes\mathcal{W}_{D^{m}}\right)  \times\left(  M^{M}\otimes\mathcal{W}%
_{D^{n}}\right) \nonumber\\
& \rightarrow\left(  M^{M}\otimes\mathcal{W}_{D^{l+m+n}}\right)  \times\left(
M^{M}\otimes\mathcal{W}_{D^{l+m+n}}\right)  \times\left(  M^{M}\otimes
\mathcal{W}_{D^{l+m+n}}\right) \nonumber\\
& =\left(  M^{M}\times M^{M}\times M^{M}\right)  \otimes\mathcal{W}%
_{D^{l+m+n}}\nonumber\\
& =\left(  \left(  M^{M}\times M^{M}\right)  \times M^{M}\right)
\otimes\mathcal{W}_{D^{l+m+n}}\text{,}\label{4.2.10}%
\end{align}
(\ref{4.2.9}) and (\ref{4.2.4}) in succession. Similarly, the right-hand side
of (\ref{4.2.0.a}) is the composition of
\begin{align}
& \left(  M^{M}\otimes\mathcal{W}_{D^{l}}\right)  \times\left(  M^{M}%
\otimes\mathcal{W}_{D^{m}}\right)  \times\left(  M^{M}\otimes\mathcal{W}%
_{D^{n}}\right) \nonumber\\
& \rightarrow\left(  M^{M}\otimes\mathcal{W}_{D^{l+m+n}}\right)  \times\left(
M^{M}\otimes\mathcal{W}_{D^{l+m+n}}\right)  \times\left(  M^{M}\otimes
\mathcal{W}_{D^{l+m+n}}\right) \nonumber\\
& =\left(  M^{M}\times M^{M}\times M^{M}\right)  \otimes\mathcal{W}%
_{D^{l+m+n}}\nonumber\\
& =\left(  M^{M}\times\left(  M^{M}\times M^{M}\right)  \right)
\otimes\mathcal{W}_{D^{l+m+n}}\text{,}\label{4.2.10'}%
\end{align}
\begin{align}
\left(  \mathrm{id}_{M^{M}}\times\mathrm{ass}_{M}\right)  \otimes
\mathrm{id}_{\mathcal{W}_{D^{l+m+n}}}  & :\left(  M^{M}\times\left(
M^{M}\times M^{M}\right)  \right)  \otimes\mathcal{W}_{D^{l+m+n}}\nonumber\\
& \rightarrow\left(  M^{M}\times M^{M}\right)  \otimes\mathcal{W}_{D^{l+m+n}%
}\label{4.2.9'}%
\end{align}
and (\ref{4.2.4}) in succession. Therefore the desired result follows from
Lemma \ref{t4.2'}.
\end{proof}

\begin{remark}
By the above associativity, we can unambiguously define such a morphism as
\begin{align*}
\mathrm{Ass}_{M}^{1,1,1}  & :\left(  M^{M}\otimes\mathcal{W}_{D}\right)
\times\left(  M^{M}\otimes\mathcal{W}_{D}\right)  \times\left(  M^{M}%
\otimes\mathcal{W}_{D}\right) \\
& \rightarrow M^{M}\otimes\mathcal{W}_{D^{3}}%
\end{align*}
to be either the composition of
\begin{align*}
\mathrm{Ass}_{M}^{1,1}\times\mathrm{id}_{M^{M}\otimes\mathcal{W}_{D}}  &
:\left(  M^{M}\otimes\mathcal{W}_{D}\right)  \times\left(  M^{M}%
\otimes\mathcal{W}_{D}\right)  \times\left(  M^{M}\otimes\mathcal{W}%
_{D}\right) \\
& =\left(  \left(  M^{M}\otimes\mathcal{W}_{D}\right)  \times\left(
M^{M}\otimes\mathcal{W}_{D}\right)  \right)  \times\left(  M^{M}%
\otimes\mathcal{W}_{D}\right) \\
& \rightarrow\left(  M^{M}\otimes\mathcal{W}_{D^{2}}\right)  \times\left(
M^{M}\otimes\mathcal{W}_{D}\right)
\end{align*}
and
\begin{align*}
\mathrm{Ass}_{M}^{2,1}  & :\left(  M^{M}\otimes\mathcal{W}_{D^{2}}\right)
\times\left(  M^{M}\otimes\mathcal{W}_{D}\right) \\
& \rightarrow M^{M}\otimes\mathcal{W}_{D^{3}}%
\end{align*}
in succession, or the composition of
\begin{align*}
\mathrm{id}_{M^{M}\otimes\mathcal{W}_{D}}\times\mathrm{Ass}_{M}^{1,1}  &
:\left(  M^{M}\otimes\mathcal{W}_{D}\right)  \times\left(  M^{M}%
\otimes\mathcal{W}_{D}\right)  \times\left(  M^{M}\otimes\mathcal{W}%
_{D}\right) \\
& =\left(  M^{M}\otimes\mathcal{W}_{D}\right)  \times\left(  \left(
M^{M}\otimes\mathcal{W}_{D}\right)  \times\left(  M^{M}\otimes\mathcal{W}%
_{D}\right)  \right) \\
& \rightarrow\left(  M^{M}\otimes\mathcal{W}_{D}\right)  \times\left(
M^{M}\otimes\mathcal{W}_{D^{2}}\right)
\end{align*}
and
\begin{align*}
\mathrm{Ass}_{M}^{1,2}  & :\left(  M^{M}\otimes\mathcal{W}_{D}\right)
\times\left(  M^{M}\otimes\mathcal{W}_{D^{2}}\right) \\
& \rightarrow M^{M}\otimes\mathcal{W}_{D^{3}}%
\end{align*}
in succession. Similarly for $\overline{\mathrm{Ass}}_{M}^{1,1,1}$.
\end{remark}

\begin{notation}
We write $i_{M}:1\rightarrow M^{M}$ for the exponential transpose of
\[
\mathrm{id}_{M}:1\times M=M\rightarrow M
\]

\end{notation}

\begin{lemma}
\label{t4.3'}We have
\[
\mathrm{ass}_{M}\circ\left(  i_{M}\times\mathrm{id}_{M^{M}}\right)
=\mathrm{ass}_{M}\circ\left(  \mathrm{id}_{M^{M}}\times i_{M}\right)
=\mathrm{id}_{M^{M}}%
\]
and
\[
\overline{\mathrm{ass}}_{M}\circ\left(  i_{M}\times\mathrm{id}_{M^{M}}\right)
=\overline{\mathrm{ass}}_{M}\circ\left(  \mathrm{id}_{M^{M}}\times
i_{M}\right)  =\mathrm{id}_{M^{M}}%
\]

\end{lemma}

\begin{notation}
We write $I_{M}^{n}:1\rightarrow M^{M}\otimes\mathcal{W}_{D^{n}}$ for the
morphism
\[
i_{M}\otimes\mathrm{id}_{\mathcal{W}_{D^{n}}}:1=1\otimes\mathcal{W}_{D^{n}%
}\rightarrow M^{M}\otimes\mathcal{W}_{D^{n}}%
\]

\end{notation}

\begin{proposition}
\label{t4.3}We have such identities as
\begin{equation}
\mathrm{Ass}_{M}^{m,n}\circ\left(  I_{M}^{m}\times\mathrm{id}_{M^{M}%
\otimes\mathcal{W}_{D^{n}}}\right)  =\mathrm{id}_{M^{M}}\otimes\mathcal{W}%
_{\left(  d_{1},...,d_{m},d_{m+1},...,d_{m+n}\right)  \in D^{m+n}%
\mapsto\left(  d_{m+1},...,d_{m+n}\right)  \in D^{n}}\label{4.3.0.a}%
\end{equation}
\begin{equation}
\mathrm{Ass}_{M}^{m,n}\circ\left(  \mathrm{id}_{M^{M}\otimes\mathcal{W}%
_{D^{m}}}\times I_{M}^{n}\right)  =\mathrm{id}_{M^{M}}\otimes\mathcal{W}%
_{\left(  d_{1},...,d_{m},d_{m+1},...,d_{m+n}\right)  \in D^{m+n}%
\mapsto\left(  d_{1},...,d_{m}\right)  \in D^{m}}\label{4.3.0.b}%
\end{equation}
\begin{equation}
\overline{\mathrm{Ass}}_{M}^{m,n}\circ\left(  I_{M}^{m}\times\mathrm{id}%
_{M^{M}\otimes\mathcal{W}_{D^{n}}}\right)  =\mathrm{id}_{M^{M}}\otimes
\mathcal{W}_{\left(  d_{1},...,d_{m},d_{m+1},...,d_{m+n}\right)  \in
D^{m+n}\mapsto\left(  d_{m+1},...,d_{m+n}\right)  \in D^{n}}\label{4.3.0.c}%
\end{equation}
and
\begin{equation}
\overline{\mathrm{Ass}}_{M}^{m,n}\circ\left(  \mathrm{id}_{M^{M}%
\otimes\mathcal{W}_{D^{m}}}\times I_{M}^{n}\right)  =\mathrm{id}_{M^{M}%
}\otimes\mathcal{W}_{\left(  d_{1},...,d_{m},d_{m+1},...,d_{m+n}\right)  \in
D^{m+n}\mapsto\left(  d_{1},...,d_{m}\right)  \in D^{m}}\label{4.3.0.d}%
\end{equation}

\end{proposition}

\begin{proof}
These follow from Lemma \ref{t4.3'} by the same token as in the proof
Proposition \ref{t4.2}.
\end{proof}

The following proposition is essentially a variant of Proposition 6 in
\S \S 3.2 of \cite{lav}.

\begin{proposition}
\label{t4.4}The composition of
\begin{equation}
\left(  M^{M}\otimes\mathcal{W}_{D}\right)  _{\mathrm{id}_{M}}\times\left(
M^{M}\otimes\mathcal{W}_{D}\right)  _{\mathrm{id}_{M}}\rightarrow\left(
M^{M}\otimes\mathcal{W}_{D}\right)  \times\left(  M^{M}\otimes\mathcal{W}%
_{D}\right)  \text{,}\label{4.4.1}%
\end{equation}
\begin{equation}
\mathrm{Ass}_{M}^{1,1}:\left(  M^{M}\otimes\mathcal{W}_{D}\right)
\times\left(  M^{M}\otimes\mathcal{W}_{D}\right)  \rightarrow M^{M}%
\otimes\mathcal{W}_{D^{2}}\text{,}\label{4.4.2}%
\end{equation}
and
\begin{equation}
M^{M}\otimes\mathcal{W}_{D^{2}}\rightarrow M^{M}\otimes\mathcal{W}%
_{D(2)}\label{4.4.3}%
\end{equation}
in succession is equal to the canonical injection
\begin{equation}
\left(  M^{M}\otimes\mathcal{W}_{D}\right)  _{\mathrm{id}_{M}}\times\left(
M^{M}\otimes\mathcal{W}_{D}\right)  _{\mathrm{id}_{M}}\rightarrow M^{M}%
\otimes\mathcal{W}_{D(2)}\label{4.4.4}%
\end{equation}
\ Similarly, the composition of (\ref{4.4.1}),
\begin{equation}
\overline{\mathrm{Ass}}_{M}^{1,1}:\left(  M^{M}\otimes\mathcal{W}_{D}\right)
\times\left(  M^{M}\otimes\mathcal{W}_{D}\right)  \rightarrow M^{M}%
\otimes\mathcal{W}_{D^{2}}\text{,}\label{4.4.5}%
\end{equation}
and (\ref{4.4.3}) in succession is equal to (\ref{4.4.4}).
\end{proposition}

\begin{proof}
Here we deal only with the former, leaving a similar treatment of the latter
to the reader. It suffices to show that the composition of (\ref{4.4.1}%
)-(\ref{4.4.3}) and
\[
\mathrm{id}_{M^{M}}\otimes\mathcal{W}_{d\in D\mapsto\left(  d,0\right)  \in
D(2)}:M^{M}\otimes\mathcal{W}_{D(2)}\rightarrow M^{M}\otimes\mathcal{W}_{D}%
\]
in succession is equal to the composition of the projection to the first
factor
\[
\left(  M^{M}\otimes\mathcal{W}_{D}\right)  _{\mathrm{id}_{M}}\times\left(
M^{M}\otimes\mathcal{W}_{D}\right)  _{\mathrm{id}_{M}}\rightarrow\left(
M^{M}\otimes\mathcal{W}_{D}\right)  _{\mathrm{id}_{M}}%
\]
followed by the canonical injection
\[
\left(  M^{M}\otimes\mathcal{W}_{D}\right)  _{\mathrm{id}_{M}}\rightarrow
M^{M}\otimes\mathcal{W}_{D}%
\]
while the composition of (\ref{4.4.1})-(\ref{4.4.3}) and
\[
\mathrm{id}_{M^{M}}\otimes\mathcal{W}_{d\in D\mapsto\left(  0,d\right)  \in
D(2)}:M^{M}\otimes\mathcal{W}_{D(2)}\rightarrow M^{M}\otimes\mathcal{W}_{D}%
\]
in succession is equal to the composition of the projection of the second
factor
\[
\left(  M^{M}\otimes\mathcal{W}_{D}\right)  _{\mathrm{id}_{M}}\times\left(
M^{M}\otimes\mathcal{W}_{D}\right)  _{\mathrm{id}_{M}}\rightarrow\left(
M^{M}\otimes\mathcal{W}_{D}\right)  _{\mathrm{id}_{M}}%
\]
followed by the canonical injection
\[
\left(  M^{M}\otimes\mathcal{W}_{D}\right)  _{\mathrm{id}_{M}}\rightarrow
M^{M}\otimes\mathcal{W}_{D}\text{,}%
\]
the details of which are left to the reader.
\end{proof}

\begin{corollary}
\label{t4.4'}The composition of (\ref{4.4.1})-(\ref{4.4.3}) followed by
\begin{equation}
\mathrm{id}_{M^{M}}\otimes\mathcal{W}_{d\in D\mapsto\left(  d,d\right)  \in
D(2)}:M^{M}\otimes\mathcal{W}_{D(2)}\rightarrow M^{M}\otimes\mathcal{W}%
_{D}\label{4.4'.1}%
\end{equation}
is factored uniquely into
\begin{equation}
\left(  M^{M}\otimes\mathcal{W}_{D}\right)  _{\mathrm{id}_{M}}\times\left(
M^{M}\otimes\mathcal{W}_{D}\right)  _{\mathrm{id}_{M}}\rightarrow\left(
M^{M}\otimes\mathcal{W}_{D}\right)  _{\mathrm{id}_{M}}\label{4.4'.2}%
\end{equation}
followed by the canonical morphism
\begin{equation}
\left(  M^{M}\otimes\mathcal{W}_{D}\right)  _{\mathrm{id}_{M}}\rightarrow
M^{M}\otimes\mathcal{W}_{D}\label{4.4'.3}%
\end{equation}
The arrow in (\ref{4.4'.2}) stands for the addition of the $k$-module $\left(
M^{M}\otimes\mathcal{W}_{D}\right)  _{\mathrm{id}_{M}}$.
\end{corollary}

\begin{lemma}
\label{t4.5'}The diagram
\[%
\begin{array}
[c]{ccccc}%
\mathcal{W}_{D} & \rightarrow & \mathcal{W}_{D^{2}} &
\begin{array}
[c]{c}%
\rightarrow\\
\rightarrow\\
\rightarrow
\end{array}
& \mathcal{W}_{D}%
\end{array}
\]
is a limit diagram in $\mathcal{W}eil_{k}$, where the left horizontal arrow
is
\[
\mathcal{W}_{\left(  d_{1},d_{2}\right)  \in D^{2}\mapsto d_{1}d_{2}\in D}%
\]
while the three right horizontal arrows are
\[
\mathcal{W}_{d\in D\mapsto\left(  d,0\right)  \in D^{2}}%
\]
\[
\mathcal{W}_{d\in D\mapsto\left(  0,d\right)  \in D^{2}}%
\]
\[
\mathcal{W}_{d\in D\mapsto\left(  0,0\right)  \in D^{2}}%
\]

\end{lemma}

\begin{proof}
The reader is referred to Proposition 7 in \S 2.2 of \cite{lav}.
\end{proof}

\begin{theorem}
\label{t4.5}The diagram
\begin{equation}%
\begin{array}
[c]{ccccc}%
M^{M}\otimes\mathcal{W}_{D} & \rightarrow & M^{M}\otimes\mathcal{W}_{D^{2}} &
\begin{array}
[c]{c}%
\rightarrow\\
\rightarrow\\
\rightarrow
\end{array}
& M^{M}\otimes\mathcal{W}_{D}\\
&  & \uparrow &  & \\
&  & \left(  M^{M}\otimes\mathcal{W}_{D}\right)  _{\mathrm{id}_{M}}%
\times\left(  M^{M}\otimes\mathcal{W}_{D}\right)  _{\mathrm{id}_{M}} &  &
\end{array}
\label{4.5.1}%
\end{equation}
is commutative, where the left horirontal arrow is
\begin{equation}
\mathrm{id}_{M^{M}}\otimes\mathcal{W}_{\left(  d_{1},d_{2}\right)  \in
D^{2}\mapsto d_{1}d_{2}\in D}\label{4.5.2}%
\end{equation}
the three right horizontal arrows are
\begin{equation}
\mathrm{id}_{M^{M}}\otimes\mathcal{W}_{d\in D\mapsto\left(  d,0\right)  \in
D^{2}}\label{4.5.3}%
\end{equation}
\begin{equation}
\mathrm{id}_{M^{M}}\otimes\mathcal{W}_{d\in D\mapsto\left(  0,d\right)  \in
D^{2}}\label{4.5.4}%
\end{equation}
\begin{equation}
\mathrm{id}_{M^{M}}\otimes\mathcal{W}_{d\in D\mapsto\left(  0,0\right)  \in
D^{2}}\label{4.5.5}%
\end{equation}
and the vertical arrow
\begin{equation}
\left(  M^{M}\otimes\mathcal{W}_{D}\right)  _{\mathrm{id}_{M}}\times\left(
M^{M}\otimes\mathcal{W}_{D}\right)  _{\mathrm{id}_{M}}\rightarrow M^{M}%
\otimes\mathcal{W}_{D^{2}}\label{4.5.6}%
\end{equation}
is the composition of
\begin{equation}
\left(  M^{M}\otimes\mathcal{W}_{D}\right)  _{\mathrm{id}_{M}}\times\left(
M^{M}\otimes\mathcal{W}_{D}\right)  _{\mathrm{id}_{M}}\rightarrow\left(
M^{M}\otimes\mathcal{W}_{D}\right)  \times\left(  M^{M}\otimes\mathcal{W}%
_{D}\right) \label{4.5.7}%
\end{equation}
\begin{align}
& <pr_{1},pr_{2},pr_{1},pr_{2}>:\left(  M^{M}\otimes\mathcal{W}_{D}\right)
\times\left(  M^{M}\otimes\mathcal{W}_{D}\right) \nonumber\\
& \rightarrow\left(  M^{M}\otimes\mathcal{W}_{D}\right)  \times\left(
M^{M}\otimes\mathcal{W}_{D}\right)  \times\left(  M^{M}\otimes\mathcal{W}%
_{D}\right)  \times\left(  M^{M}\otimes\mathcal{W}_{D}\right) \label{4.5.8}%
\end{align}
\begin{align}
\mathrm{Ass}_{M}^{1,1,1,1}  & :\left(  M^{M}\otimes\mathcal{W}_{D}\right)
\times\left(  M^{M}\otimes\mathcal{W}_{D}\right)  \times\left(  M^{M}%
\otimes\mathcal{W}_{D}\right)  \times\left(  M^{M}\otimes\mathcal{W}%
_{D}\right) \nonumber\\
& \rightarrow M^{M}\otimes\mathcal{W}_{D^{4}}\label{4.5.9}%
\end{align}
\begin{equation}
\mathrm{id}_{M^{M}}\otimes\mathcal{W}_{\left(  d_{1},d_{2}\right)  \in
D^{2}\mapsto\left(  d_{1},d_{2},-d_{1},-d_{2}\right)  \in D^{4}}:M^{M}%
\otimes\mathcal{W}_{D^{4}}\rightarrow M^{M}\otimes\mathcal{W}_{D^{2}%
}\label{4.5.10}%
\end{equation}
in succession. Therefore, there exists a unique arrow
\begin{equation}
\left(  M^{M}\otimes\mathcal{W}_{D}\right)  _{\mathrm{id}_{M}}\times\left(
M^{M}\otimes\mathcal{W}_{D}\right)  _{\mathrm{id}_{M}}\rightarrow M^{M}%
\otimes\mathcal{W}_{D}\label{4.5.11}%
\end{equation}
making the triangle
\begin{equation}%
\begin{array}
[c]{ccc}%
M^{M}\otimes\mathcal{W}_{D} & \rightarrow & M^{M}\otimes\mathcal{W}_{D^{2}}\\
& \nwarrow & \uparrow\\
&  & \left(  M^{M}\otimes\mathcal{W}_{D}\right)  _{\mathrm{id}_{M}}%
\times\left(  M^{M}\otimes\mathcal{W}_{D}\right)  _{\mathrm{id}_{M}}%
\end{array}
\label{4.5.12}%
\end{equation}
commutative. Furthermore, the arrow in (\ref{4.5.11})\ is factored uniquely
into
\begin{equation}
\left(  M^{M}\otimes\mathcal{W}_{D}\right)  _{\mathrm{id}_{M}}\times\left(
M^{M}\otimes\mathcal{W}_{D}\right)  _{\mathrm{id}_{M}}\rightarrow\left(
M^{M}\otimes\mathcal{W}_{D}\right)  _{\mathrm{id}_{M}}\label{4.5.13}%
\end{equation}
followed by
\begin{equation}
\left(  M^{M}\otimes\mathcal{W}_{D}\right)  _{\mathrm{id}_{M}}\rightarrow
M^{M}\otimes\mathcal{W}_{D}\label{4.5.14}%
\end{equation}

\end{theorem}

\begin{proof}
It is easy to see that the composition of the arrow in (\ref{4.5.6}) and any
one of the three arrows in (\ref{4.5.3})-(\ref{4.5.5}) results in the same
arrow. Therefore the second statement follows directly from Lemma \ref{t4.5'}.
The last statement is easy to verify.
\end{proof}

\begin{notation}
The arrow in (\ref{4.5.13})\ is denoted by $L_{M}$.
\end{notation}

\begin{theorem}
\label{t4.6}The $k$-module $\left(  M^{M}\otimes\mathcal{W}_{D}\right)
_{\mathrm{id}_{M}}$ endowed with $L_{M}$\ as a Lie bracket is a Lie
$k$-algebra object in the category $\mathcal{K}$
\end{theorem}

\begin{proof}
The proof could be obtained by reformulating our proof of Proposition 16 above
and Sections 5 and 6 in \cite{nishi1}. The Jacobi identity is dealt with in
detail in \cite{nishi3}.
\end{proof}

\end{document}